\documentclass[11pt,reqno,a4paper]{amsart}
\usepackage{euscript}

\newtheorem{theorem}{Theorem}
\newtheorem{proposition}{Proposition}
\newtheorem{corollary}{Corollary}
\newtheorem{lemma}{Lemma}
\newtheorem{remark}{Remark}
\newtheorem{example}{Example}
\renewcommand{\epsilon}{\varepsilon}
\renewcommand{\phi}{\varphi}

\DeclareMathOperator{\Ima}{Im}

\def\Z{\mathbb{Z}}
\def\R{\mathbb{R}}
\def\cA{\EuScript{A}}

\def\Id{\text{\rm Id}}

\begin{document}

\title[Quasi-shadowing for partially hyperbolic dynamics]{Quasi-shadowing for partially hyperbolic dynamics on Banach spaces}

\begin{abstract}
A partially hyperbolic dynamical system is said to have the quasi-shadowing property if every pseudotrajectory  can be shadowed by a sequence of points $(x_n)_{n\in \Z}$ such that $x_{n+1}$ is obtained from the image of $x_n$ by moving it by a small factor in the central direction. In the present paper, we prove that a  small nonlinear perturbation of a partially dichotomic sequence of (not necessarily invertible) linear operators acting on an arbitrary  Banach space has  the quasi-shadowing property. We also get obtain a continuous time version of this result. As an application of our main result, we prove that a certain class of partially dichotomic sequences of linear operators is stable up to the movement in the  central direction.
\end{abstract}

\author{Lucas Backes}
\address{\noindent Departamento de Matem\'atica, Universidade Federal do Rio Grande do Sul, Av. Bento Gon\c{c}alves 9500, CEP 91509-900, Porto Alegre, RS, Brazil.}
\email{lucas.backes@ufrgs.br} 

\author{Davor Dragi\v cevi\'c}
\address{Department of Mathematics, University of Rijeka, Croatia}
\email{ddragicevic@math.uniri.hr}

\date{\today}

\maketitle

\maketitle

\section{Introduction}

Let $M=(M,d)$ be a metric space and let $(F_n)_{n\in \Z}$ be a sequence of maps acting on $M$. Given $\delta>0$, a sequence of points $(y_n)_{n\in \Z}$ of $M$ satisfying 
\begin{equation}\label{eq: pseudo intro}
d(y_{n+1}, F_n(y_n))<\delta \text{ for every } n\in \Z,
\end{equation}
is said to be a \emph{$\delta$-pseudotrajectory} for the nonautonomous dynamics given by
\begin{equation}\label{eq: diff eq intro}
x_{n+1}=F_n(x_n), \quad n\in \Z.
\end{equation}
We say that the system \eqref{eq: diff eq intro} has the \emph{shadowing property} if for every $\varepsilon>0$ there exists $\delta>0$ such that for any $\delta$-pseudotrajectory $(y_n)_{n\in \Z}$ of \eqref{eq: diff eq intro}, there exists a sequence of points $(x_n)_{n\in \Z}$ in $M$ satisfying \eqref{eq: diff eq intro} such that
\begin{equation}\label{eq: shadow intro}
d(x_n, y_n)<\varepsilon \text{ for every } n\in \Z.
\end{equation}
In other words, any pseudotrajectory can be approximated (in the sense of \eqref{eq: shadow intro}) by a true trajectory. 

The shadowing property has proved to be a very powerful tool in many situations as, for instance, when dealing with problems concerned with topological stability and construction of symbolic dynamics (see \cite{Bow75}). Consequently, the problem of describing classes of systems that exhibit this property turned out to an important direction of the research in the field of dynamical systems.

Based on works of Poincar\'e and its predecessors, Smale \cite{S67} introduced in the 60's the notion of \emph{(uniform) hyperbolicity} in order to provide a mathematical framework for the rigorous study of dynamical systems that exhibit sensitive dependence on initial conditions. It turns out that uniformly hyperbolic dynamical systems (both with discrete and continuous time) have the  shadowing property~\cite{An70,Bow75}. In fact, it was shown that in certain settings,  the notions of hyperbolicity and (Lipschitz) shadowing are actually (almost) equivalent. Indeed, Pilyugin and Tikhomirov~\cite{PT10} proved that a diffeomorphism on a compact Riemannian manifold $M$ has the so-called Lipschitz shadowing property if and only if it is structurally stable. While structurally stable diffeomorphisms are not necessarily Anosov (i.e. uniformly hyperbolic on entire $M$), it should be noted that they exhibit uniform  hyperbolicity on the set of nonwandering points. 
A similar result to that in~\cite{PT10} but dealing with the nonautonomous linear dynamics was established in~\cite[Proposition 3]{BD}. Indeed, it was proved in~\cite{BD} that if $M$ is a finite-dimensional Banach space and if maps $F_n$ are linear  then~\eqref{eq: diff eq intro} has the Lipschitz shadowing property if and only if it admits an exponential trichotomy (which means that it is hyperbolic on both $\Z^+$ and $\Z^-$, although not necessarily on entire line $\Z$).
It turns out that the situation in the infinite-dimensional setting is much more complicated.  Indeed,  in~\cite{BCDMP} the authors deal with the situation when $M$ is an infinite-dimensional Banach space and consider the autonomous setting when $F_n=A$ for $n\in \Z$, where $A$ is some invertible bounded linear operator on $M$. They give explicit examples in which $A$ is not hyperbolic but~\eqref{eq: diff eq intro} nevertheless has the shadowing property.

Our objective in this paper is study the shadowing property beyond (uniform) hyperbolicity in the nonautonomous context. To this end, we look at \emph{partially hyperbolic systems}. A first observation is that one cannot expect to recover the shadowing property, at least in its full strength, for general partially hyperbolic systems. In fact, it was observed in \cite{BDT} that the shadowing property is not verified (not even generically) for partially hyperbolic diffeomorphisms which are robustly transitive (see also \cite{AD07,YY00} for examples of some large classes of non-uniformly hyperbolic systems that do not satisfy the shadowing property). On the other hand, as we are going to see in the sequel, it is possible to get a weaker version of the shadowing property called \emph{quasi-shadowing}. In the terminology of the above paragraphs, we say that \eqref{eq: diff eq intro} has the quasi-shadowing property if for every $\varepsilon>0$ there exists $\delta>0$ so that for any $\delta$-pseudotrajectory $(y_n)_{n\in \Z}$ there exists a ``quasi-trajectory" $(x_n)_{n\in \Z}$ of \eqref{eq: diff eq intro} satisfying \eqref{eq: shadow intro}. By $(x_n)_{n\in \Z}$ being a quasi-trajectory of \eqref{eq: diff eq intro} we mean that it is a trajectory of the system up to moving it by a small factor in the central direction, i.e. $x_{n+1}$ is obtained from $F_n(x_n)$ by shifting it by a small factor in the central direction (precise definitions are postponed to Section \ref{sec: setup}).

The notion of quasi-shadowing has already been explored for some classes of partially hyperbolic systems. For instance, in \cite{BB16, CRV, HZZ15,KT13,ZZ17} versions of the quasi-shadowing property were established for partially hyperbolic diffeomorphisms acting on compact manifolds. More recently some of these results were extended to partially hyperbolic flows~\cite{LZ20}. In this work, we deal with not necessarily invertible  \emph{nonautonomous dynamics} acting on \emph{infinite dimensional} spaces.  More precisely, starting with a linear dynamics
\begin{equation}\label{eq: intr}
x_{m+1}=A_m x_m \quad m\in \Z, 
\end{equation}
where the sequence $(A_m)_{m\in \Z}$ admits a \emph{partial dichotomy}, we prove that a small \emph{nonlinear} perturbation of~\eqref{eq: intr} has the quasi-shadowing property. We also obtain a continuous time version of this result. Moreover, our general approach using \emph{Banach sequence spaces} allow us to get various versions of the quasi-shadowing property, meaning that the error allowed in the pseudotrajectory \eqref{eq: pseudo intro} and in the ``shadowing" given by \eqref{eq: shadow intro} can be taken to be  small according to various norms (such as  $l^p$ norm or the $c_0$ norm), simply by taking the appropriate sequence space (see Remark \ref{remark: general approach}). As an application of our main result we prove that a certain class of partial dichotomic sequence of linear maps is stable up to moving it in the central direction (see Section \ref{sec: quasi-stability}). The proof of our main result has an analytic flavor and consists basically of showing that a certain operator is a contraction when acting in an appropriate space. Then, using the fixed point of this operator we are able to construct the quasi-trajectory we are looking for. These arguments are inspired by our previous work on nonautonomous shadowing~\cite{BD19, BD}, which in turn are inspired by some classical
analytic approaches to shadowing~\cite{CLP, MS}.

\section{Preliminaries}\label{P}
\subsection{Banach sequence spaces}
In this subsection we recall some basic
definitions and properties from the theory of Banach sequence spaces. The material is taken from~\cite{DD, Sasu}, where the reader can also find more details. 

Let $\mathcal{S}(\Z)$ be the set of all sequences $\mathbf{s}=(s_n)_{n\in \Z}$ of real numbers. We say that a linear subspace $B\subset \mathcal{S}(\Z)$ is a \emph{normed sequence space} (over $\Z$) if there exists a norm $\lVert \cdot \rVert_B \colon B \to \R_0^+$ such that if $\mathbf{s}'=(s_n')_{n\in \Z}\in B$ and $\lvert s_n\rvert \le \lvert s_n'\rvert$ for $n\in \Z$, then $\mathbf{s}=(s_n)_{n\in \Z}\in B$ and $\lVert \mathbf{s}\rVert_B \le \lVert \mathbf{s}'\rVert_B$. If in addition $(B, \lVert \cdot \rVert_B)$ is complete, we say that $B$ is a \emph{Banach sequence space}.

Let $B$ be a Banach sequence space over $\Z$. We say that $B$ is \emph{admissible} if:
\begin{enumerate}
\item
$\chi_{\{n\}} \in B$ and $\lVert \chi_{\{n\}}\rVert_B >0$ for $n\in \Z$, where $\chi_A$ denotes the characteristic function of the set $A\subset \Z$;
\item
for each $\mathbf{s}=(s_n)_{n\in \Z}\in B$ and $m\in \Z$, the sequence $\mathbf{s}^m=(s_n^m)_{n\in \Z}$ defined by $s_n^m=s_{n+m}$ belongs to $B$ and  $\lVert \mathbf{s}^m \rVert_B = \lVert \mathbf{s}\rVert_B$.
\end{enumerate}
Note that it follows from the definition that for each admissible Banach space $B$ over $\Z$, we have that $\lVert \chi_{\{n\}}\rVert_B=\lVert \chi_{\{0\}}\rVert_B$ for each $n\in \Z$. Throughout this paper we will assume for the sake of simplicity  that $\lVert \chi_{\{0\}}\rVert_B=1$.

We recall some  explicit examples of admissible  Banach sequence spaces over $\Z$ (see~\cite{DD,Sasu}).

\begin{example}\label{ex1}
The set \[l^\infty =\bigg{\{} \mathbf{s}=(s_n)_{n\in \Z} \in \mathcal{S} (\Z): \sup_{n\in \Z} \lvert s_n \rvert < \infty \bigg{\}}\] is an admissible Banach sequence space when equipped with the norm $\lVert \mathbf{s} \rVert =\sup_{n\in \Z} \lvert s_n \rvert$.
\end{example}

\begin{example}\label{ex2}
The set \[c_0= \bigg{\{} \mathbf{s}=(s_n)_{n\in \Z} \in \mathcal{S} (\Z): \lim_{\lvert n\rvert \to \infty} \lvert s_n\rvert=0 \bigg{\}}  \] is an admissible Banach sequence space when equipped with the norm $\lVert \cdot \rVert$ from Example~\ref{ex1}.
\end{example}

\begin{example}\label{ex3}
For each $p\in [1, \infty )$, the set \[ l^p= \bigg{\{} \mathbf{s}=(s_n)_{n\in \Z} \in \mathcal{S}(\Z): \sum_{n\in \Z} \lvert s_n \rvert^p <\infty \bigg{\}} \] is an admissible Banach sequence space when equipped with the norm \[\lVert \mathbf{s} \rVert= \bigg{(}\sum_{n\in \Z} \lvert s_n \rvert^p \bigg{)}^{1/p}.\]
\end{example}

\begin{example}[Orlicz sequence spaces]
Let $\phi \colon (0, +\infty) \to (0, +\infty]$ be a nondecreasing nonconstant left-continuous function. We set $\psi(t)=\int_0^t \phi(s) \, ds$  for $t\ge 0$. Moreover, for each $\mathbf s=(s_n)_{n\in \Z} \in \mathcal S(\Z)$, let $M_\phi (\mathbf s)=\sum_{n\in \Z} \psi(\lvert s_n \rvert)$. Then
\[
B=\bigg{\{} \mathbf s \in \mathcal S(\Z) : M_\phi (c \mathbf s)<+\infty \ \text{for some} \ c>0 \bigg{\}}
\]
is an admissible Banach sequence space when equipped with the norm
\[
\lVert \mathbf s \rVert=\inf \bigl\{ c>0 : M_\phi ( \mathbf s / c ) \le 1 \bigr\}.
\]
\end{example}

\subsection{Banach spaces associated to Banach sequence spaces}
Let us now introduce sequence spaces that will play important role in our arguments. 
Let $X$ be an arbitrary Banach space and $B$ any Banach sequence space over $\Z$ with norm $\lVert \cdot \rVert_B$. Set
\[
X_B:=\bigg{\{} \mathbf x=(x_n)_{n\in \Z} \subset X: (\lVert x_n\rVert)_{n\in \Z}\in B \bigg{\}}.
\]
Finally, for $\mathbf x=(x_n)_{n\in \Z} \in X_B$ we define 
\begin{equation}\label{nn}
\lVert \mathbf x\rVert_B:=\lVert (\lVert x_n\rVert)_{n\in \Z}\rVert_B.
\end{equation}
\begin{remark}
We emphasize that in~\eqref{nn} we slightly abuse the notation since norms on  $B$ and $X_B$ are denoted in the same way. However, this will cause no confusion since in the rest of the paper we will deal with spaces $X_B$.
\end{remark}
\begin{example}
Let $B=l^\infty$ (see Example~\ref{ex1}). Then, 
\[
X_B=\bigg{\{} \mathbf x=(x_n)_{n\in \Z} \subset X: \sup_{n\in \Z} \lVert x_n\rVert<\infty \bigg{\}}.
\]
\end{example}
The proof of the following result is straightforward (see~\cite{DD, Sasu}).
\begin{proposition}
$(X_B, \lVert \cdot \rVert_B)$ is a Banach space. 
\end{proposition}

\subsection{Partial dichotomy} \label{sec: ph}
In this subsection we introduce the concept of partial dichotomy as well as  some related notation.

Let $(A_m)_{m\in \Z}$ be a sequence of bounded linear operators on $X$. For $m, n\in \Z$, set
\[
\cA(m,n)=\begin{cases}
A_{m-1}\cdots A_n & \text{if $m>n$,} \\
\Id & \text{if $m=n$.} 
\end{cases}
\]
We say that the sequence $(A_m)_{m\in \Z}$ admits a \emph{partial exponential dichotomy} if there exist projections $P_n^i$ for $n\in \Z$ and $i\in \{1, 2, 3\}$ satisfying 
\begin{equation}\label{Proj}
P_n^1+P_n^2+P_n^3=\Id, \quad A_nP_n^i=P_{n+1}^i A_n,
\end{equation}
for $n\in \Z$ and $i\in \{1, 2, 3\}$ such that the operator 
\[
A_n \rvert_{\Ima P_n^2} \colon \Ima P_n^2 \to \Ima P_{n+1}^2
\]
is invertible for $n\in \Z$ and there exist constants $D, b, d>0$ such that 
\begin{equation}\label{d1}
\lVert \cA(m,n)P_n^1 \rVert \le De^{-d(m-n)} \quad \text{for $m\ge n$} 
\end{equation}
and
\begin{equation}\label{d2}
\lVert \cA(m, n)P_n^2\rVert \le De^{-b(n-m)} \quad \text{for $m\le n$,}
\end{equation}
where
\[
\cA(m, n)=(\cA(n,m)\rvert_{\Ima P_m})^{-1} \colon \Ima P_n^2 \to \Ima P_m^2,
\]
for $m<n$.
\begin{remark}
We note that the classical notion of an exponential dichotomy is a special case of the notion of partial exponential dichotomy and corresponds  to the case when $P_n^3=0$ for $n\in \Z$.
\end{remark}
\begin{remark}
The above introduced notion of a partial exponential dichotomy is inspired by the classical notion of  a partial hyperbolicity introduced by Brin and Pesin~\cite{BP}. However, we stress that in constrast to the notion of the  partial hyperbolicity, we don't require that the rate of the  contraction/expansion of vectors in $\Ima P_n^3$ by the
action of the dynamics forward in time is dominated by the contraction along $\Ima P_n^1$ or by the expansion along $\Ima P_n^2$. In fact, we made no assumption about the asymptotic behaviour of the dynamics along $\Ima P_n^3$.
\end{remark}
Given a partially  dichotomic sequence $(A_m)_{m\in \Z}$,  for $n\in \Z$ we set
\[
E_n^s:=\Ima P_n^1, \ E_n^u:=\Ima P_n^2 \ \text{and} \ E_n^c:=\Ima P_n^3.
\]
Moreover, we also introduce $X_B^{s,u}$ as a subspace of $X_B$ that consists of all $\mathbf x=(x_n)_{n\in \Z}\in X_B$ such that $x_n\in E_n^{s,u}:= E_n^s\oplus E_n^u$ for each $n\in \Z$. Obviously, $X_B^{s,u}$ is closed. Similarly, we consider
\[
X_B^c:=\{\mathbf x=(x_n)_{n\in \Z}\in X_B: \text{$x_n\in E_n^c$ for $n\in \Z$}\}.
\]
Again, $X_B^c$ is a closed subspace of $X_B$. Observe that each $\mathbf x=(x_n)_{n\in \Z} \in X_B$ can be written uniquely as 
\[
\mathbf x=\mathbf x^c+\mathbf x^{s,u},
\]
where $\mathbf x^c\in X_B^c$ and $\mathbf x^{s,u}\in X_B^{s,u}$. Indeed, $\mathbf x^c=(x_n^c)_{n\in \Z}$ and $\mathbf x^{s,u}=(x_n^{s,u})_{n\in \Z}$ are given by
\[
x_n^c=P_n^3 x_n \quad \text{and} \quad  x_n^{s,u}=x_n-x_n^c,  \quad  \text{for $n\in \Z$.}
\]
We also introduce an adapted norm $\lVert \cdot \rVert_B'$ on $X_B$ defined by
\[
\lVert \mathbf x\rVert_B'=\max \{\lVert \mathbf x^c\rVert_B, \lVert \mathbf x^{s,u}\rVert_B \}, \quad \text{for $\mathbf x\in X_B$.}
\]
\begin{lemma}
We have that 
\begin{equation}\label{norms}
\frac{1}{1+2D}\lVert \mathbf x\rVert_B' \le \lVert \mathbf x\rVert_B \le 2\lVert \mathbf x\rVert_B',  \quad \text{for $\mathbf x\in X_B$.}
\end{equation}
Thus, the norms $\lVert \cdot \rVert_B$ and $\lVert \cdot \rVert_B'$ are equivalent. 
\end{lemma}

\begin{proof}
Observe that
\[
\lVert \mathbf x\rVert_B=\lVert \mathbf x^c+\mathbf x^{s,u}\rVert_B \le \lVert \mathbf x^c\rVert_B+\lVert \mathbf x^{s,u}\rVert_B \le 2\lVert \mathbf x\rVert_B'.
\]
On the other hand, we first note that it follows from~\eqref{d1} and~\eqref{d2} that $\lVert P_n^i\rVert \le D$ for $n\in \Z$ and $i=1,2$. Thus, \[
\lVert P_n^3\rVert =\lVert \Id-P_n^1-P_n^2\rVert \le 1+2D \quad \text{for $n\in \Z$,}
\]
and consequently 
\[
\lVert x_n^c\rVert \le (1+2D)\lVert x_n\rVert \quad \text{and} \quad \lVert x_n^{s,u}\rVert \le 2D \lVert x_n \rVert,
\]
for $n\in \Z$ and $\mathbf x=(x_n)_{n\in \Z}\in X_B$. It follows that 
\[
\lVert \mathbf x^c\rVert_B \le (1+2D)\lVert \mathbf x\rVert_B \quad \text{and} \quad \lVert \mathbf x^{s,u}\rVert_B \le 2D \lVert \mathbf x\rVert_B.
\]
Therefore 
\[
\lVert \mathbf x\rVert_B' \le (1+2D)\lVert \mathbf x\rVert_B,
\]
and the proof of the lemma is completed. 
\end{proof}

\section{Main result}\label{MR}
\subsection{Setup} \label{sec: setup}
Let $B$ be an admissible Banach sequence space,  $X$  a  Banach space and $(A_m)_{m\in \Z}$ a sequence of  bounded linear operators on $X$ that admits a partial  exponential dichotomy.
Furthermore, let $f_n \colon X\to X$, $n\in \Z$, be a sequence of  maps such that there exists $c>0$ satisfying
\begin{equation}\label{fg}
\lVert f_n(x)-f_n(y)\rVert \le c\lVert x-y\rVert,
\end{equation}
for each $n\in \Z$ and $x, y\in X$.

We consider a nonautonomous and nonlinear dynamics given by 
\begin{equation}\label{nnd}
x_{n+1}=F_n(x_n),  \quad n\in \Z,
\end{equation}
where \[F_n:=A_n+f_n.\]

Let us now recall some notation introduced in~\cite{BD19}.
Given $\delta >0$, a sequence $(y_n)_{n\in \Z} \subset X$ is said to be an $(\delta, B)$-\emph{pseudotrajectory} for~\eqref{nnd} if  
$(y_{n+1}-F_n(y_n))_{n\in \Z}\in X_B$ and 
\begin{equation}\label{pseudo1}
\lVert (y_{n+1}-F_n(y_n))_{n\in \Z} \rVert_B \le \delta. 
\end{equation}
\begin{remark}
When $B=l^\infty$ (see Example~\ref{ex1}), condition~\eqref{pseudo1} reduces to 
\[
\sup_{n\in \Z}  \lVert y_{n+1}-F_n(y_n) \rVert \le \delta.
\]
The above requirement represents a usual definition of a pseudotrajectory in the context of smooth dynamics (see~\cite{Pal00, Pil99}). 
\end{remark}

We say that~\eqref{nnd} has an \emph{$B$-quasi-shadowing property} if for every $\varepsilon>0$ there exists $\delta >0$ so that for every $(\delta, B)$-pseudotrajectory $(y_n)_{n\in \Z}$, there is a sequence  $\mathbf z=(z_n)_{n\in \Z}\in X_B$ with $\lVert \mathbf z\rVert_B' \le \varepsilon$ such that for every $n\in \Z$,
\begin{equation}\label{pseudo}
x_{n+1}=F_n(x_n) + z^c_{n+1},
\end{equation}
where $x_n=y_n+z_n^{s,u}$. Observe that, in particular, $x_{n+1}$ is obtained from $F_n(x_n)$ by moving it by a factor smaller than $\varepsilon$ in the central direction and, moreover, $\|(x_n)_{n\in \Z}-(y_n)_{n\in \Z}\|_B\leq \varepsilon$. Informally, $(x_n)_{n\in \Z}$ is a ``quasi-trajectory" of \eqref{nnd} that ``shadows" $(y_n)_{n\in \Z}$. Furthermore, if there exists $L>0$ such that $\delta$ can be chosen as $\delta=L\epsilon$, we say that~\eqref{nnd} has the \emph{$B$-Lipschitz quasi-shadowing property}.

\subsection{Quasi-shadowing for perturbations of partial dichotomic sequences}
We start with an auxiliary result which is a straightforward consequence of Theorems 1 and 2 of \cite{BD}.
\begin{theorem}\label{admph}
Assume that a sequence $(A_m)_{m\in \Z}$ is partially dichotomic and let $B$ be an arbitrary admissible Banach sequence space. Then,  there exists an operator $\mathbb A^{s,u}\colon X_B^{s,u}\to X_B^{s,u}$ with the property that for $\mathbf x=(x_n)_{n\in \Z}, \mathbf y=(y_n)_{n\in \Z}\in X_B^{s,u}$, the following  properties are equivalent:
\begin{enumerate}
\item $\mathbb A^{s,u}\mathbf y=\mathbf x$;
\item for each $n\in \Z$,
\[
x_n-A_{n-1}x_{n-1}=y_n.
\]
\end{enumerate}
   In fact, $\mathbb A^{s,u}$ is given by 
\[
(\mathbb A^{s,u} \mathbf y)_n=\sum_{m=-\infty}^n \cA(n,m)P_m^1y_m-\sum_{m=n+1}^\infty \cA(n, m)P_m^2y_m, 
\]
for $n\in \Z$ and $\mathbf y=(y_n)_{n\in \Z}\in X_B^{s,u}$.
\end{theorem}
Let us consider the operator $G\colon X_B\to X_B$ given by
\[
G\mathbf x=-\mathbf x^c+\mathbb A^{s,u} \mathbf x^{s,u}.
\]
By $\lVert G\rVert$ we will denote the operator norm of $G$ induced by the norm $\lVert \cdot \rVert_B'$ on $X_B$.

\begin{theorem}\label{theo: main}
Assume that 
\begin{equation}\label{cG}
4cD(1+2D)\lVert G\rVert<1. 
\end{equation}
Then, the  system~\eqref{nnd} has the  $B$-Lipschitz quasi-shadowing property. 
\end{theorem}

\begin{proof} 
Set
\[
(S (\mathbf x))_n=g_{n-1}(x_{n-1}^{s,u}), 
\]
for $n\in \Z$ and $\mathbf x=(x_n)_{n\in \Z}\in X_B$, where $g_n:X\to X$ is given by
$$g_n(x)=f_n(x+y_n)-f_n(y_n)+F_n(y_n)-y_{n+1}.$$
Furthermore, let 
\[
\Phi (\mathbf x)=GS(\mathbf x), \quad \mathbf x\in X_B.
\]
It follows from~\eqref{fg} that for $\mathbf x=(x_n)_{n\in \Z}$ and $\mathbf z=(z_n)_{n\in \Z}$ in $X_B$ and $n\in \Z$, we have 
\[
\begin{split}
 \lVert f_n(y_n+x_n^{s,u})-f_n(y_n+z_n^{s,u})\rVert 
&\le c \lVert x_n^{s,u}-z_n^{s,u}\rVert \\
&=c\lVert (P_n^1+P_n^2)(x_n-z_n)\rVert \\
&\le 2cD \lVert x_n-z_n\rVert,
\end{split}
\]
and hence
\[
\lVert S(\mathbf x)-S(\mathbf z)\rVert_B \le 2cD \lVert \mathbf x-\mathbf z\rVert_B.
\]
Consequently, \eqref{norms} implies that 
\[
\lVert S(\mathbf x)-S(\mathbf z)\rVert_B' \le 4cD(1+2D) \lVert \mathbf x-\mathbf z\rVert_B',
\]
and therefore
\begin{equation}\label{contraction}
\lVert \Phi (\mathbf x)-\Phi(\mathbf z)\rVert_B' \le 4cD(1+2D) \lVert G\rVert \cdot \lVert \mathbf x-\mathbf z\rVert_B', \quad \text{for $\mathbf x, \mathbf z\in X_B$.}
\end{equation}
On the other hand, observe that 
\[
\lVert S(\mathbf 0)\rVert_B' \le (1+2D)\lVert S(\mathbf 0)\rVert_B\le (1+2D)\delta,
\]
and therefore using~\eqref{contraction} we have that 
\[
\begin{split}
\lVert \Phi (\mathbf x)\rVert_B' &\le \lVert \Phi(\mathbf 0)\rVert_B'+\lVert \Phi(\mathbf x)-\Phi(\mathbf 0)\rVert_B' \\
&\le (1+2D)\delta \lVert G\rVert+4cD(1+2D) \lVert G\rVert \cdot \lVert \mathbf x\rVert_B'.
\end{split}
\]
We conclude that by setting 
\[
\delta=\frac{1-4cD(1+2D) \lVert G\rVert}{(1+2D) \lVert G\rVert}\epsilon
\]
and
\[
 D(\mathbf 0, \epsilon):=\{\mathbf x\in X_B: \lVert \mathbf x\rVert_B'\le \epsilon \},
\]
we have that \[\Phi ( D(\mathbf 0, \epsilon))\subset  D(\mathbf 0, \epsilon).\] This together with~\eqref{cG} and~\eqref{contraction} implies that $\Phi$ is a contraction on $ D(\mathbf 0, \epsilon)$. 
Hence, there exist a unique $\mathbf z\in D(\mathbf 0, \epsilon)$ such that $\Phi(\mathbf z)=\mathbf z$, that is, $GS(\mathbf z)=\mathbf z$.  Letting $\mathbf{w}=S(\mathbf z)$, we have that 
\[
-\mathbf {w}^c+\mathbb A^{s,u}\mathbf{w}^{s,u}=\mathbf z. 
\]
This in particular implies that $-\mathbf{w}^c=\mathbf z^c$. Moreover, Theorem~\ref{admph} implies that for each $n\in \Z$,
\[
\begin{split}
z_n+w_n^c-A_{n-1}(z_{n-1}+w_{n-1}^c) &=z_n-z_n^c-A_{n-1}(z_{n-1}-z_{n-1}^c) \\
&=w_n^{s,u} \\
&=w_n-w_n^c \\
&=w_n+z_n^c,
\end{split}
\]
which easily implies that 
\begin{equation}\label{1134}
y_n+z_n^{s,u}=z_n^c+F_{n-1}(y_{n-1}+z_{n-1}^{s,u}),
\end{equation}
for each $n\in \Z$. We define $\mathbf x=(x_n)_{n\in \Z}$  by \[x_n=y_n+z_n^{s,u},\quad  n\in \Z.\] Then, it follows from~\eqref{1134} that~\eqref{pseudo} holds. Consequently, since $\lVert \mathbf z\rVert_B'\le \epsilon$, we conclude that~\eqref{nnd} has the  $B$-Lipschitz quasi-shadowing property. 
\end{proof}
In the sequel, we will also need the following lemma. 
\begin{lemma}\label{lemma: aux unique}
Assume that~\eqref{cG} holds and let $L>0$ be the constant given by the $B$-Lipschitz quasi-shadowing property. Given $\varepsilon>0$, take $\delta=L\varepsilon$ and fix a $(\delta, B)$-pseudotrajectory $(y_n)_{n\in \Z}$ of \eqref{nnd}. Suppose that $\mathbf z=(z_n)_{n\in \Z}\in X_B$ is a sequence with $\lVert \mathbf z\rVert_B' \le \varepsilon$ which satisfies
\begin{displaymath}
x_{n+1}=F_n(x_n) + z^c_{n+1},
\end{displaymath}
for every $n\in \Z$, where $x_n=y_n+z_n^{s,u}$. Then, $\mathbf z=(z_n)_{n\in \Z}$ is a fixed point of the operator $\Phi$ introduced in the proof of Theorem \ref{theo: main}.
\end{lemma}
\begin{proof}
We start observing that
\begin{displaymath}
\begin{split}
\Phi(\mathbf z) &=GS(\mathbf z) \\
&=-S(\mathbf z)^c+\mathbb{A}^{s,u}(S(\mathbf z)^{s,u})
\end{split}
\end{displaymath}
and
\begin{displaymath}
\begin{split}
\left( S(\mathbf z) \right)_n &=  g_{n-1}(z_{n-1}^{s,u}) \\
&=f_{n-1}(z_{n-1}^{s,u}+y_{n-1})-f_{n-1}(y_{n-1})+F_{n-1}(y_{n-1})-y_n\\
&=f_{n-1}(x_{n-1})+A_{n-1}y_{n-1}-y_n\\
&=f_{n-1}(x_{n-1})+A_{n-1}x_{n-1}-A_{n-1}z^{s,u}_{n-1}-y_n\\
&=x_n-z_n^c-A_{n-1}z^{s,u}_{n-1}-y_n\\
&=z^{s,u}_n-z_n^c-A_{n-1}z^{s,u}_{n-1}.
\end{split}
\end{displaymath}
Consequently, 
\begin{displaymath}
\begin{split}
\left(-S(\mathbf z)^c \right)_n &=-\left(z^{s,u}_n-z_n^c-A_{n-1}z^{s,u}_{n-1}\right)^c\\
&=z_n^c,
\end{split}
\end{displaymath}
for $n\in \Z$.
Hence, by using Theorem~\ref{admph}, we have that 
\[
\begin{split}
&\left(\mathbb{A}^{s,u}(S(\mathbf z)^{s,u})\right)_n  \\
&=\sum_{m=-\infty}^n \cA(n,m)P_m^1(S(\mathbf z)^{s,u})_m-\sum_{m=n+1}^\infty \cA(n, m)P_m^2(S(\mathbf z)^{s,u})_m \\
&=\sum_{m=-\infty}^n \cA(n,m)(z^{s}_m-A_{m-1}z^{s}_{m-1})-\sum_{m=n+1}^\infty \cA(n, m)(z^{u}_m-A_{m-1}z^{u}_{m-1})\\
&= z_n^s+z^u_n\\
&=z_n^{s,u},
\end{split}
\]
for $n\in \Z$.
Therefore, 
$
\Phi(\mathbf z) = \mathbf z
$
and the proof of the lemma is completed.
\end{proof}

\begin{corollary}\label{Cor}
Suppose that~\eqref{cG} holds. Then, the sequence $\mathbf z=(z_n)_{n\in \Z}$ given by the quasi-shadowing property is unique.
\end{corollary}
\begin{proof}
From Lemma \ref{lemma: aux unique} we know that any sequence $\mathbf z=(z_n)_{n\in \Z}$ given by the quasi-shadowing property is a fixed point of the operator $\Phi$. Consequently, since $\Phi$ is a contraction on $D(\mathbf 0, \epsilon)=\{\mathbf x\in X_B: \lVert \mathbf x\rVert_B'\le \epsilon \}$  its fixed point is unique in $D(\mathbf 0, \epsilon)$, and  the desired conclusion follows. 
\end{proof}

In the case when the sequence $(A_m)_{m\in \Z}$ admits an exponential dichotomy we can say more. More precisely, we have the following result first established in~\cite[Theorem 4.]{BD}.

\begin{corollary}
Assume that the sequence $(A_m)_{m\in \Z}$ admits an exponential dichotomy and that~\eqref{cG} holds. Then, there exists $L>0$ with the property that for each $\varepsilon >0$ and every $(\delta, B)$-pseudotrajectory $\mathbf y=(y_n)_{n\in \Z}$ with $\delta=L\varepsilon$, there exists a solution 
$\mathbf x=(x_n)_{n\in \Z}$ of~\eqref{nnd} such that $\lVert \mathbf x-\mathbf y\rVert_B' \le \epsilon$. Moreover, $\mathbf x$ is unique. 
\end{corollary}

\begin{proof}
The desired conclusion follows directly from Theorem~\ref{theo: main} and Corollary~\ref{Cor} taking into account that $E_n^c=\{0\}$ for $n\in \Z$.
\end{proof}

\begin{remark} \label{remark: general approach}
Observe that our general approach allow us to get various versions of the quasi-shadowing property simply by considering different types of Banach sequence spaces. For instance, by taking $B=l^\infty$ as in Example \ref{ex1} we get a quasi-shadowing version of the usual shadowing property. By taking $B=l^p$ as in Example \ref{ex3} we get a quasi-shadowing version of the $l^p$-shadowing property. By taking $B=c_0$ as in Example \ref{ex2} we get a quasi-shadowing version of the asymptotic shadowing property and so on.
\end{remark}

\section{Quasi-Stability of partially dichotomic sequences} \label{sec: quasi-stability}

As an application of our results in this section  we prove that, up to moving it in the central direction, a class of  partially dichotomic sequences of linear operators acting on an arbitrary Banach space is stable under nonlinear perturbations.

We say that a sequence $(A_m)_{m\in \Z}$ of bounded and invertible  linear operators on $X$ admits a \emph{strong partial exponential dichotomy} if there exist projections $P_n^i$ for $n\in \Z$ and $i\in \{1, 2, 3\}$ satisfying~\eqref{Proj} and there exist constants 
\[D>0, \quad 0\le a<b, \quad \text{and} \quad 0\le c<d, 
\]
such that~\eqref{d1} and \eqref{d2} hold and, in addition, 
\begin{equation}\label{d3}
\lVert \cA(m,n)P_n^3\rVert \le De^{a(m-n)}  \quad \text{for $m\ge n$,}
\end{equation}
and
\[
 \lVert \cA(m,n)P_n^3\rVert \le De^{c(n-m)} \quad \text{for $m\le n$,}
\]
where
\[
\cA(m, n)=(\cA(n,m))^{-1}, \quad \text{for $m<n$.}
\]

Let $(A_m)_{m\in \Z}$ be a sequence of bounded and invertible  linear operators on $X$ that admits a strong partial exponential dichotomy. Furthermore, assume that  \[\sup_{m\in \Z}\{\lVert A_m\rVert,\lVert A_m^{-1}\rVert \} <\infty.\] Associated to these parameters by Theorem \ref{theo: main} (applied to $B=l^\infty$ and $f_n\equiv 0$), consider $\varepsilon>0$ sufficiently small and $\delta=L\varepsilon>0$. Let $(f_n)_{n\in \Z}$ be a sequence of maps $f_n\colon X\to X$ satisfying \eqref{fg} with $c$ sufficiently small and such that
\[\lVert f_n\rVert_{\sup} \le \delta \quad  \text{for each $n\in \Z$,} \] 
where for  a map $g\colon X\to X$ we set
\[
\lVert g\rVert_{\sup}:=\sup \{\lVert g(x)\rVert: x\in X\}.
\]
We consider the difference equation
\begin{equation}\label{eq: rec GH}
y_{n+1}=F_n(y_n), \quad n\in \Z,
\end{equation}
where $F_n:=A_n+f_n$. By decreasing $c$ (if necessary), we have that $F_n$ is a homeomorphism for each $n\in \Z$ (see~\cite{BD19}). Then, we have the following result. 

\begin{theorem}\label{theo: GH} 
There are continuous maps $h_m \colon X \to X$ and $\tau_m:X\to E_m^c$, $m\in \Z$, such that for each $m\in \Z$,
\begin{equation}\label{eq:gh1}
h_{m+1}\circ F_m=A_m \circ h_m +\tau_{m+1}\circ F_m
\end{equation}
with
\begin{equation}\label{eq:gh2}
\lVert h_m -\Id\rVert_{\sup}\leq \epsilon \text{ and } \lVert \tau_m\rVert_{\sup}\leq \epsilon.
\end{equation}
Moreover, $h_m(x)-x\in E^{s,u}_m$ for all $m\in \mathbb{Z}$ and $x\in X$.
\end{theorem}
Recall that the nonautonomous systems $x_{m+1}=A_mx_m$, $m\in \Z$, and $x_{m+1}=F_m(x_m)$, $m\in \Z$, are said to be \emph{topologically conjugated} if there exists a sequence of homeomorphisms $(h_m)_{m\in \Z}$ such that
$$h_{m+1}\circ F_m=A_m \circ h_m \text{ for every } m\in \Z. $$
Moreover, a system is said to be \emph{stable} if it is topologically conjugated to any small (according to some appropriate topology) perturbation of itself. So, what Theorem \ref{theo: GH} is saying is that, under the previous assumptions, the system $x_{m+1}=A_mx_m$, $m\in \Z$, is \emph{quasi-stable}: it is stable except for a small deviation in the central direction (the ``$\tau_{m+1}\circ F_m$" part in \eqref{eq:gh1}) and the fact that the maps $h_m$ are only continuous.

\begin{proof} 
Let us fix   $m\in \Z$. Given $y\in X$, we consider the sequence $\mathbf y=(y_n)_{n\in \Z}$ given by
$y_n=\mathcal{F}(n, m)y$ for $n\in \Z$, where
\[
\mathcal{F}(n,m)=
\begin{cases}
F_{n-1}\circ \ldots \circ F_m & \text{if $n>m$,}\\
\Id & \text{if $n=m$,}\\
F_n^{-1}\circ \ldots \circ F_{m-1}^{-1} & \text{if $n<m$.}
\end{cases}
\]
Then, $\mathbf y$ is a solution of~\eqref{eq: rec GH}. Moreover,
\begin{displaymath}
\sup_{n\in \Z}\lVert y_{n+1}-A_n y_n\rVert=\sup_{n\in \Z} \lVert f_n(y_n)\rVert \le
\delta.
\end{displaymath} 
In particular, $\mathbf y=(y_n)_{n\in \Z}$ is a $(\delta, l^\infty)$-pseudotrajectory for the difference equation \[x_{n+1}=A_nx_n, \quad  n\in \Z.\] Hence, it follows from Theorem~\ref{theo: main} (applied to the case when $B=l^\infty$ and $f_n\equiv 0$) and Corollary \ref{Cor} that there is a unique sequence $\mathbf z=(z_n)_{n\in \Z}$ with $\|\mathbf z\|'_B\leq \varepsilon$ such that \[x_{n+1}=A_nx_n+z^c_{n+1} \quad  \text{for $n\in \Z$,}\] where $x_n=y_n+z^{s,u}_n$. Set
\[
 h_m(y)=h_m(y_m):=x_m
\]
and  
$$\tau_m(y)=\tau_m(y_m):=z^c_m.$$

It is easy to verify that~\eqref{eq:gh1} holds. Since $\|\mathbf z\|'_B\leq \varepsilon$, we conclude that~\eqref{eq:gh2} holds. Moreover, by definition, $h_m(y)-y=z_m^{s,u}\in E^{s,u}_m$ for every $y\in X$. 

 It remains to show that $h_m$ and $\tau_m$ are continuous maps for every $m\in \Z$. For the sake of simplicity we deal with the case when $m=0$. The argument for $m\neq 0$ is completely analogous. 
Let $\Phi:X_B\to X_B$ be the operator introduced in the proof of Theorem~\ref{theo: main} associated to the sequence $\mathbf y=(y_n)_{n\in \Z}$. We recall that 
\begin{displaymath}
\Phi(\mathbf z)=\mathbf z.
\end{displaymath}
By recalling the definition of $h_m$ and $\tau_m$, the above equality can be rewritten as 
\begin{displaymath}
\Phi((h_m(y_m)-y_m+\tau_m(y_m))_{m\in \Z})=(h_m(y_m)-y_m+\tau_m(y_m))_{m\in \Z}.
\end{displaymath}
Consequently, from the definition of $\Phi$ and Theorem \ref{admph} we obtain that
\[
\begin{split}
&(h_m(y_m)-y_m+\tau_m(y_m))_{m\in \Z} \\
&=\Phi((h_m(y_m)-y_m+\tau_m(y_m))_{m\in \Z})\\
&=G S((h_m(y_m)-y_m+\tau_m(y_m))_{m\in \Z})\\
&=\Bigg( -(g_{m-1}(h_{m-1}(y_{m-1})-y_{m-1}))^c \\
&+ \sum_{k=-\infty}^m \cA(m,k)P_k^1g_{k-1}(h_{k-1}(y_{k-1})-y_{k-1})\\
& -\sum_{k=m+1}^\infty \cA(m, k)P_k^2g_{k-1}(h_{k-1}(y_{k-1})-y_{k-1}) \Bigg)_{m\in \Z} .
\end{split}
\]
In particular, since $h_m(y_m)-y_m\in E^{s,u}_m$ and $\tau_m(y_m)\in E^c_m$, it follows that 
\begin{displaymath}
\tau_m(y_m)=-(g_{m-1}(h_{m-1}(y_{m-1})-y_{m-1}))^c
\end{displaymath}
and consequently, $\tau_m$ is continuous whenever $h_{m-1}$ is. So, all we have to do is to prove that $h_0$ is continuous.

Let $(w_0^j)_{j\in \mathbb{N}}$ be an arbitrary sequence in $X$ converging to $y_0\in X$ with $\|w^j_0-y_0\|<\varepsilon$ for every $j\in \mathbb{N}$. If $h_0(w^j_0)\xrightarrow{j\to \infty}h_0(y_0)$ then we are done. So, let us assume that $h_0(w^j_0)\not\rightarrow h_0(y_0)$. In particular, restricting ourselves to a subsequence, if necessary, we may assume that $\|h_0(w^j_0)- h_0(y_0)\|>4\gamma$ for every $j\in \mathbb{N}$ and some $\gamma >0$. Thus,
\begin{displaymath}
\|(h_0(y_0)-y_0) -(h_0(w^j_0)-w^j_0)\|=\|(h_0(y_0)-h_0(w^j_0)) +(w^j_0-y_0)\|>3\gamma
\end{displaymath}
for every $j\gg 0$. Consequently, recalling that $h_0(y_0)-y_0\in E^{s,u}_0$ and $h_0(w^j_0)-w^j_0\in E^{s,u}_0$ for every $j$, it follows that
\[
\begin{split}
&\|(h_0(y_0)-h_0(w^j_0))^{s,u} +(w^j_0-y_0)^{s,u}\| \\
&= \|\left((h_0(y_0)-h_0(w^j_0)) +(w^j_0-y_0)\right)^{s,u}\| \\
&=\|(h_0(y_0)-h_0(w^j_0)) +(w^j_0-y_0)\| >3\gamma
\end{split}
\]
for every $j\gg 0$ and thus, since $\|w^j_0-y_0\|\to 0$ when $j\to +\infty$, we have that 
\begin{displaymath}
\|(h_0(y_0)-h_0(w^j_0))^{s,u}\|>2\gamma
\end{displaymath}
for every $j\gg 0$. In particular, restricting our selves to a subsequence, if necessary, we may assume that
\begin{displaymath}
\|(h_0(y_0)-h_0(w^j_0))^{s}\| >\gamma \text{ or } \|(h_0(y_0)-h_0(w^j_0))^{u}\| >\gamma
\end{displaymath}
for every $j\in \mathbb{N}$. Suppose we are in the second case.

For every $j\in \mathbb{N}$, let $(w^j_m)_{m\in \Z}$ be the sequence given by $w^j_m=\mathcal{F}(m,0)w^j_0$. From the continuity of $F_m$ it follows that for every $j\in \mathbb{N}$ there exists $N_j\in \mathbb{N}$ so that $\|y_m-w^j_m\|<\varepsilon$ for every $|m|\leq N_j$. Moreover, $N_j$ can be taken so that $N_j\to \infty$ when $j\to \infty$. 

Now, on the one hand, using \eqref{eq:gh1} we have that for every $m\in \mathbb{N}$,
\begin{equation*}
\begin{split}
\|h_m(y_m)-h_m(w^j_m)\|&= \bigg{\|} \mathcal{A}(m,0)(h_0(y_0)-h_0(w^j_0)) +\sum_{i=1}^m\mathcal{A}(m,i)(\tau_i(y_i)-\tau_i(w^j_i)) \bigg{\|}.
\end{split}
\end{equation*}
Consequently, for every $m\in \mathbb{N}$, we have that (using~\eqref{d1}, \eqref{d2} and~\eqref{d3})
\[
\begin{split}
&\|h_m(y_m)-h_m(w^j_m)\| \\
&\geq \| \mathcal{A}(m,0)(h_0(y_0)-h_0(w^j_0))\| - \bigg{\|}\sum_{i=1}^m\mathcal{A}(m,i)(\tau_i(y_i)-\tau_i(w^j_i)) \bigg{\|} \\
&\geq \| \mathcal{A}(m,0)(h_0(y_0)-h_0(w^j_0))^u\| -\| \mathcal{A}(m,0)(h_0(y_0)-h_0(w^j_0))^s\| \\
&-\| \mathcal{A}(m,0)(h_0(y_0)-h_0(w^j_0))^c\|- \bigg{\|}\sum_{i=1}^m\mathcal{A}(m,i)(\tau_i(y_i)-\tau_i(w^j_i)) \bigg{\|} \\
&\geq \frac{1}{D}e^{bm}\|(h_0(y_0)-h_0(w^j_0))^u\| -De^{-dm}\|(h_0(y_0)-h_0(w^j_0))^s\| \\
&-De^{am}\|(h_0(y_0)-h_0(w^j_0))^c\|-\sum_{i=1}^m De^{a(m-i)}\|\tau_i(y_i)-\tau_i(w^j_i)\|. \\
\end{split}
\]
Recalling that $h_0(w^j_0)=w^j_0+(z^j_{0})^{s,u}$ with $\|(z^j_{0})^{s,u}\|\leq\varepsilon$ for every $j\in \mathbb{N}$ and that $(w^j_0)_{j\in \mathbb{N}}$ converges to $y_0$, we conclude that $\|h_0(y_0)-h_0(w^j_0)\|$ is uniformly bounded. In particular, there exists $C>0$ so that $\|(h_0(y_0)-h_0(w^j_0))^s\| <C$ and $\|(h_0(y_0)-h_0(w^j_0))^c\|<C$ for every $j\in \mathbb{N}$. Moreover, since $\|(h_0(y_0)-h_0(w^j_0))^u\|>\gamma$ and $\|\tau_i(y_i)-\tau_i(w^j_i)\|\leq \|\tau_i(y_i)\|+\|\tau_i(w^j_i)\|\le 2\varepsilon$, it follows from the previous expression that
\begin{equation*}
\begin{split}
\|h_m(y_m)-h_m(w^j_m)\| &\geq \frac{1}{D}e^{bm}\gamma -De^{-dm}C-De^{am}C-mDe^{am}2\varepsilon, \\
\end{split}
\end{equation*}
for every $m\in \mathbb{N}$ and $j\in \mathbb{N}$. Thus, since $b>a$, there exists $m_0\in \mathbb{N}$ so that for every $m\geq m_0$, 
\begin{displaymath}
\|h_m(y_m)-h_m(w^j_m)\| \geq 10\varepsilon 
\end{displaymath}
for every $j\in \mathbb{N}$. Fix $j\gg 0$ such that $N_j\geq m_0$.
In particular,
\begin{equation}\label{eq: unbounded h_m}
\|h_{N_j}(y_{N_j})-h_{N_j}(w^j_{N_j})\|\geq 10\varepsilon.
\end{equation}
On the other hand, by the choice of $N_j$ we have that 
\begin{equation*}\label{eq: bound h_m} 
\begin{split}
\|h_m(y_m)-h_m(w^j_m)\|&\leq \|h_m(y_m)-y_m\|+\|y_m-w^j_m\|\\ 
&+\|w^j_m-h_m(w^j_m)\|\\
&\leq \varepsilon +\varepsilon +\varepsilon \\
&=3\varepsilon,
\end{split}
\end{equation*}
for every $|m|\leq N_j$. This together with~\eqref{eq: unbounded h_m} yields a contradiction. The case when $\|(h_0(y_0)-h_0(w^j_0))^{s}\| >\gamma$ for every $j\in \mathbb N$ can be treated analogously by taking backward iterates. Consequently, $h_0(w^j_0)\xrightarrow{j\to \infty}h_0(y_0)$ and $h_0$ is continuous as claimed.
\end{proof}

\begin{remark}
Assuming $X$ is finite dimensional one can easily see that the maps $h_m$ given by the previous theorem are surjective. In fact, in this setting, any bounded continuous perturbation of the identity is surjective. Indeed,  let $h:X\to X$ be a continuous map so that $\|h\|_{sup} \le N$, where $N>0$. We are going to observe that $\Id+h$ is surjective. Given $y\in X$ let us consider the continuous map $H:X\to X$ given by $H(x)=y-h(x)$. Since $h$ is bounded by $N$, it follows that $H$ maps the closed ball of radius $N$ around $y$ into itself. Consequently, by Brouwer's fixed point theorem, $H$ has a fixed point inside that ball. In particular, there is $x\in X$ so that $H(x)=x$ which is equivalent to $x+h(x)=y$ proving that $\Id+h$ is surjective.
\end{remark}

\begin{remark}
In the case when $P_n^3=0$ for $n\in \Z$ (i.e. when $(A_m)_{m\in \Z}$ admits an exponential dichotomy), we have that $\tau_m=0$ and that $h_m$ is a homeomorphism for each $m\in \Z$ (see~\cite[Theorem 5]{BD19}).  Hence, in this setting, Theorem~\ref{theo: GH} essentially reduces to the so-called nonautonomous Grobman-Hartman theorem first established (for finite-dimensional and continuous time dynamics) by Palmer~\cite{P}. For some new results related to nonautonomous linearization devoted to the situations when conjugacies $h_m$ exhibit higher regularity, we refer to~\cite{DZZ1, DZZ2} together with the discussion and references therein. 
\end{remark}

\section{The case of continuous time}

  We consider a nonlinear differential equation
\begin{equation}\label{ndc}
x'=A(t)x+f(t,x),
\end{equation}
where $A$ is a continuous map from $\R$ to the space of all bounded linear operators  on $X$ satisfying
\[
N:=\sup_{t\in \R} \lVert A(t)\rVert<\infty,
\]
 and $f\colon \R \times X\to X$ is a continuous map. We assume that $f(\cdot,0)=0$  and that there exists $c>0$ such that
\[
\lVert f(t,x)-f(t,y)\rVert \le c\lVert x-y\rVert \quad \text{for $t\in \R$ and $x, y\in X$.}
\]
 We consider the associated linear equation
\begin{equation}\label{ldc}
x'=A(t)x.
\end{equation}
Let $T(t,s)$ be the (linear) evolution family associated to~\eqref{ldc}.  We will assume that~\eqref{ldc} admits a \emph{partial exponential dichotomy}, i.e. that  there exists a family of projections $P^i(s)$,  $i\in \{1, 2,3\}$, $s\in \R$ on $X$  satisfying
\[
P^1(s)+P^2(s)+P^3(s)=\Id, \quad T(t,s)P^i(s)=P^i(t)T(t,s),
\]
for $t,s \in \R$, $i\in \{1, 2, 3\}$ and there exist constants $D, b, d>0$ such that
\[
\lVert T(t,s)P^1(s) \rVert \le De^{-d(t-s)} \quad \text{for $t\ge s$} 
\]
and
\[
\lVert T(t,s)P^2(s)\rVert \le De^{-b(s-t)} \quad \text{for $t\le s$.}
\]
We recall that the nonlinear evolution family associated with~\eqref{ndc} is given by
\[
U(t,s)x=T(t,s)x+\int_s^t T(t,\tau)f(\tau, U(\tau, s)x)\, d\tau,
\]
for $x\in X$ and $t, s\in \R$. Furthemore, set
\[
f_n(x)=\int_n^{n+1} T(n+1, \tau)f(\tau, U(\tau, n)x)\, d\tau, \quad \text{for $x\in X$ and $n\in \Z$.}
\]
Finally, let
\begin{equation}\label{Fn}
A_n=T(n+1, n) \quad \text{and} \quad F_n=A_n+f_n=U(n+1, n), 
\end{equation}
for $n\in \Z$. We observe that the sequence $(A_n)_{n\in \Z}$ admits a partial exponential dichotomy. Finally, we recall the adapted norm $\lVert \cdot \rVert_{\infty}'$ that corresponds to $B=l^\infty$ (see Example~\ref{ex1}). In addition, in this case we will denote $X_B$ by $X_\infty$.

The following is the main result of this section. 
\begin{theorem}
For a sufficiently small $c>0$, there exists $L>0$ with the following property: for any $\varepsilon >0$ and a differentiable function $y\colon \R \to X$  such that 
\[
\sup_{t\in \R}\lVert y'(t)-A(t)y(t)-f(t, y(t))\rVert \le \delta:=L\varepsilon,
\]
there exist $x\colon \R \to X$ such that:
\begin{enumerate}
\item $x\rvert_{(n, n+1)}$ is a solution of~\eqref{ndc} on $(n, n+1)$ for each $n\in \Z$;
\item 
\[
\sup_{t\in \R}\lVert x(t)-y(t)\rVert \le \varepsilon;
\]
\item there exists $\mathbf z=(z_n)_{n\in \Z} \in X_\infty$ satisfying $\lVert \mathbf z\rVert_\infty' \le \varepsilon$ such that~\eqref{pseudo} holds with $x_n=x(n)$, $n\in \Z$, where $F_n$ is given by~\eqref{Fn}.
\end{enumerate}
\end{theorem}

\begin{proof}
By arguing as in the proof of~\cite[Theorem 6.]{BD} and using Theorem~\ref{theo: main}, one can easily show that for a sufficiently small $c>0$, \eqref{nnd} has an $l^\infty$-Lipschitz quasi-shadowing property. Let us denote the associated constant by $L'$.
Furthermore, it is also proved in the proof of~\cite[Theorem 6.]{BD} that there exists $t>0$ such that for any differentiable $y\colon \R \to X$ such that
\begin{equation}\label{pseudo2}
\sup_{t\in \R}\lVert y'(t)-A(t)y(t)-f(t, y(t))\rVert \le \delta,
\end{equation}
we have that the sequence $(y_n)_{n\in \Z}$ given by $y_n=y(n)$, $n\in \Z$ is an $(t\delta, l^\infty)$-pseudotrajectory for~\eqref{nnd}. Set
\[
L:=\frac{1}{(1+t/L')e^{N+c}}>0.
\]

For $\varepsilon >0$, set $\delta=L\varepsilon $ and fix a differentiable map $y\colon \R\to X$ satisfying~\eqref{pseudo2}. By the preceding discussion, the sequence $(y_n)_{n\in \Z}$ given by $y_n=y(n)$, $n\in \Z$ is an $(t\delta, l^\infty)$-pseudotrajectory for~\eqref{nnd}. Hence, Theorem~\ref{theo: main} implies that there exists 
 $\mathbf z=(z_n)_{n\in \Z} \in X_\infty$ satisfying $\lVert \mathbf z\rVert_\infty' \le \frac{t\delta}{L'}\le \varepsilon$ such that~\eqref{pseudo} holds, where $x_n=y_n+z_n^{s,u}$, $n\in \Z$. We define $x\colon \R \to X$ by
\[
x(t)=U(t, n)x_n, \quad \text{for $t\in [n, n+1)$, $n\in \Z$.}
\]
Then, $x$ satisfies the first and the third assertion in the statement of the theorem.  Define $h\colon \R \to X$ by 
\[
h(t)=y'(t)-A(t)y(t)-f(t,y(t)), \quad t\in \R. 
\]
 Observe that for $n\in \Z$ and $t\in [n, n+1)$ we have  that
\[
\begin{split}
& \lVert x(t)-y(t)\rVert \le \lVert x_n-y_n\rVert \\
&+ \bigg{\lVert} \int_n^t (A(s)(x(s)-y(s))+f(s, x(s))-f(s, y(s))-h(s))\, ds \bigg{\rVert} \\
&\le \delta \bigg{(}1+\frac{t}{L'}\bigg{)}+(N+c) \int_t^n \lVert x(s)-y(s)\rVert\, ds.
\end{split}
\]
Hence, Gronwall's lemma implies that 
\[
\sup_{t\in \R}\lVert x(t)-y(t)\rVert \le  \delta \bigg{(}1+\frac{t}{L'}\bigg{)} e^{N+c}=\varepsilon.
\]
The proof of the theorem is completed.
\end{proof}


\medskip{\bf Acknowledgements.}
We would like to thank the anonymous referee for his/hers constructive comments that helped us improve our paper. D.D would like to thank  Ken Palmer for useful discussion.
 L.B. was partially supported by a CNPq-Brazil PQ fellowship under Grant No. 306484/2018-8. D.D. was supported in part by Croatian Science Foundation under the project
IP-2019-04-1239 and by the University of Rijeka under the projects uniri-prirod-18-9 and uniri-pr-prirod-19-16.



\begin{thebibliography}{11}

\bibitem{AD07}
F. Abdenur and L. D\'iaz, \emph{Pseudo-orbit shadowing in the $C^1$ topology}, Discrete Contin. Dyn. Syst.
\textbf{17} (2007), 223--245.

\bibitem{An70}
D. Anosov, {On a class of invariant sets of smooth dynamical systems}, [in Russian]. In: Proc. 5th Int. Conf. on Nonl. Oscill. 2. Kiev (1970), 39--45.
 
 
\bibitem{BD19}
L. Backes and D. Dragi\v cevi\' c, \emph{Shadowing for nonautonomous dynamics}, Advanced Nonlinear Studies, \textbf{19} (2019), 425--436.

\bibitem{BD}
L. Backes and D. Dragi\v cevi\' c, \emph{Shadowing for infinite dimensional dynamics and exponential trichotomies}, Proc. Roy. Soc. Edinburgh Sect. A, in press, 
https://doi.org/10.1017/prm.2020.42


\bibitem{BCDMP}
N. Bernardes Jr., P.R. Cirilo, U. B. Darji, A. Messaoudi and E. R. Pujals, \emph{Expansivity and Shadowing in linear dynamics}, J. Math. Anal. Appl.  \textbf{461} (2018), 796--816.


\bibitem{BP}
M. Brin and  Ya. Pesin,  \emph{Partially hyperbolic dynamical systems}, Izv. Akad. Nauk SSSR Ser. Mat. \textbf{38} (1974), 170--212.

\bibitem{BB16}
D. Bohnet and C. Bonatti, \emph{Partially hyperbolic diffeomorphisms with uniformly center foliation: The quotient dynamics}, Ergod. Theory Dyn. Syst., \textbf{36} (2016), 1067--1105.

\bibitem{BDT} 
Ch. Bonatti, L. Diaz and G. Turcat, \emph{There is no shadowing lemma for partially hyperbolic dynamics}, C. R. Acad. Sci. Paris Ser. I Math. \textbf{330} (2000), 587.


\bibitem{Bow75}
R.~Bowen, \emph{Equilibrium States and the Ergodic Theory of Anosov Diffeomorphisms},
Lect. Notes in Math. 470. Springer-Verlag (1975).

\bibitem{CRV}
A. Castro, F. Rodrigues and P. Varandas, \emph{Leafwise shadowing property for partially hyperbolic diffeomorphisms}, Dyn. Syst., in press. 


\bibitem{CLP}
S. N. Chow, X. B. Lin and K. J. Palmer, \emph{A shadowing lemma with applications to
semilinear parabolic equations}, SIAM J. Math. Anal. \textbf{20} (1989), 547–557.

\bibitem{DD}
D. Dragi\v cevi\' c, \emph{Admissibility, a general type of Lipschitz shadowing and structural stability}, Comm. Pure Appl. Anal. \textbf{14} (2015), 861--880.

\bibitem{DZZ1}
D. Dragi\v cevi\' c, W. Zhang and W. Zhang, \emph{Smooth linearization of nonautonomous difference equations with a nonuniform dichotomy}, Math. Z. \textbf{292} (2019), 1175--1193.

\bibitem{DZZ2}
D. Dragi\v cevi\' c, W. Zhang and W. Zhang, \emph{Smooth linearization of nonautonomous differential equations with a nonuniform dichotomy}, Proc. Lond. Math. Soc. \textbf{121} (2020), 32--50.





\bibitem{HZZ15}
H. Hu, Y. Zhou and Y Zhu, \emph{Quasi-shadowing for partially hyperbolic diffeomorphisms}. Ergodic Theory and Dynamical Systems, \textbf{35} (2015), 412--430.

\bibitem{KT13}
S. Kryzhevich and S. Tikhomirov, \emph{Partial hyperbolicity and central shadowing}, Discrete Contin. Dyn. Syst., \textbf{33} (2013), 2901--2909.

\bibitem{LZ20}
Z. Li and Y. Zhou, \emph{Quasi-shadowing for partially hyperbolic flows}, Discrete \& Continuous Dynamical Systems  A, \textbf{40} (2020), 2089--2103. 

\bibitem{MS}
K. R. Meyer and G. R. Sell, \emph{An Analytic Proof of the Shadowing Lemma}, Funkcialaj
Ekvacioj \textbf{30} (1987), 127–133.

\bibitem{P}
K. Palmer, \emph{A generalization of Hartman’s linearization theorem}, J. Math. Anal. Appl.
\textbf{41} (1973), 753--758.

\bibitem{Pal00}
K. Palmer, \emph{Shadowing in Dynamical Systems. Theory and Applications}, Kluwer, Dordrecht, 2000.

\bibitem{Pil99}
S. Yu. Pilyugin, \emph{Shadowing in Dynamical Systems}, Lecture Notes in Mathematics, vol.1706, Springer-Verlag, Berlin, 1999.


\bibitem{PT10}
S.Yu. Pilyugin and S.B. Tikhomirov, \emph{Lipschitz shadowing implies structural stability}, Nonlinearity \textbf{23} (2010), 2509--2515.


\bibitem{Sasu}
A. L. Sasu, \emph{Exponential dichotomy and dichotomy radius for difference equations}, J. Math. Anal. Appl.  \textbf{344} (2008), 906--920.


\bibitem{S67}
S. Smale, \emph{Differentiable dynamical systems}, Bulletin of the American Mathematical Society. \textbf{73} (1967) 747--817.

\bibitem{YY00}
G.-C. Yuan and J.A. Yorke, \emph{An open set of maps for which every point is absolutely non-shadowable}, Proc. Amer. Math. Soc. \textbf{128} (2000), 909--919.
 
\bibitem{ZZ17}
F. Zhang and Y. H. Zhou, \emph{On the limit quasi-shadowing property}, Discrete Contin. Dyn. Syst., \textbf{37} (2017), 2861--2879.
 

\end{thebibliography}
\end{document}